\documentclass[12pt]{article}
\usepackage{amssymb}
\usepackage{amsfonts}
\usepackage{amsmath}
\usepackage[usenames]{color}
\usepackage{mathrsfs}
\usepackage{amsfonts}
\usepackage{amssymb,amsmath}
\usepackage{CJK}
\usepackage{cite}
\usepackage{cases}
\usepackage{amsthm}

\def\cl{\centerline}

\def\vs{\vspace*}
\def\W{\mathcal{W}}

\def\V{\mathcal{V}}
\def\G{\mathcal{G}}
\def\H{\mathcal{H}}
\def\D{\mathfrak{D}}
\def\Z{\mathbb{Z}}
\def\g{\mathfrak{g}}
\def\N{\mathbb{N}}
\def\K{\mathcal{K}}
\def\A{\mathcal{A}}

\def\C{\mathbb{C}}

\def\ni{\noindent}

\numberwithin{equation}{section}
\newtheorem{theo}{Theorem}[section]
\newtheorem{defi}[theo]{Definition}

\newtheorem{lemm}[theo]{Lemma}

\newtheorem{prop}[theo]{Proposition}
\newtheorem{clai}{Claim}

\newtheorem{remark}[theo]{Remark}

\begin{document}
\begin{center}
\cl{\large\bf \vs{6pt} Harish-Chandra modules  over the high rank}
\cl{\large\bf $W$-algebra  $W(2,2)$}
\cl{ Haibo Chen}
\end{center}

{\small
\parskip .005 truein
\baselineskip 3pt \lineskip 3pt

\noindent{{\bf Abstract:}
In this paper, using the theory of $\A$-cover  developed in \cite{B1,BF1}, we    completely   classify  all simple Harish-Chandra modules over the high rank $W$-algebra $W(2,2)$.  As a byproduct, we   obtain the classification of   simple Harish-Chandra
modules over the  classical $W$-algebra $W(2,2)$ studied in \cite{LGZ,CLW,GLZ1}.
\vs{5pt}

\ni{\bf Key words:}
 high rank
$W$-algebra $W(2,2)$,   Harish-Chandra  module, weight module.}

\ni{\it Mathematics Subject Classification (2020):} 17B10, 17B65,  17B68.}
\parskip .001 truein\baselineskip 6pt \lineskip 6pt
\section{Introduction}
In the representation theory of infinite-dimensional Lie algebras, there is  a  very important class of weight modules  called Harish-Chandra  modules (namely, weight modules with finite-dimensional weight spaces).  The classification of simple Harish-Chandra modules over the Virasoro algebra (also called $N=0$ superconformal algebra),  conjectured  by  Kac (see \cite{K}), was given in \cite{M2}.
Combined \cite{MP} with \cite{S0},  a new method was presented  to   obtain   this classification.
After that, a lot  of general versions of the Virasoro algebra
have been investigated by some authors. Those include, but are not limited to,
  the generalized Virasoro algebra (see, e.g., \cite{LZ2,Ma3,S1,S2,S,SZ,PS,GLZ}), the
 generalized Heisenberg-Virasoro algebra (see \cite{GLZ0,LG,LZ1}), the $W$-algebra $W(2,2)$ (see \cite{LGZ,CLW,GLZ1}),  the loop-Virasoro algebra (see \cite{GLZ1}),
 and
so on.
To   classify  all simple Harish-Chandra modules over the
Lie algebra  $W_n$ of vector fields on  $n$-dimensional torus, Billig and Futorny developed a new technique called $\A$-cover theory in \cite{B1,BF1}. The  result gained here  was a generalized version of   Mathieu's classification theorem for the Virasoro algebra. From then on, the $\A$-cover theory was used in some other Lie (super)algebras (see, e.g., \cite{BF2,CLW,XL,BFIK}).

  The $W$-algebra $W(2,2)$  was introduced  in \cite{ZD} by Zhang and Dong for  investigated the classification of simple vertex
operator algebras generated by   two weight $2$ vectors.
 The centerless $W$-algebra $W(2,2)$  $\overline{\W}[\Z]$  can be  obtained  from the point of view of  non-relativistic analogues of the  conformal field theory. By using the ``non-relativistic limit''   on a pair of commuting
algebras  $\mathfrak{vect}(S^1)\oplus \mathfrak{vect}(S^1)$  (see \cite{RU}) via a group contraction, one has the following generators
\begin{eqnarray*}
&&L_m=-t^{m+1}\frac{d}{dt}-(m+1)t^my\frac{d}{dy}-(m+1)\sigma t^m-m(m+1)\eta t^{m-1}y,
\\&&W_m=-t^{m+1}\frac{d}{dy}-(m+1)\eta t^m,
\end{eqnarray*}
where $m\in\Z$,     $\sigma$  and $\eta$ are respectively the scaling dimension  and a  free parameter.
Then {\em the centerless $W$-algebra $W(2,2)$}   is  a Lie algebra with the basis $\{L_m,W_m\mid m\in\Z\}$ and the non-vanishing commutators as follows $$[L_m,L_{m^\prime}]=(m^\prime-m)L_{m+m^\prime},\ [L_m,W_{m^\prime}]=(m^\prime-m)W_{m+m^\prime}$$
 for $m,m^\prime\in\Z$.
  It  is  an infinite-dimensional extension
of an algebra called either non-relativistic  or   conformal Galilei algebra $\mathrm{CGA}(1)\cong\langle L_{\pm1,0},W_{\pm1,0}\rangle$ (see \cite{G}).
Basically, the relationship with conformal algebras makes  it widely studied
 in string theory (see \cite{BS}).
Furthermore, $\overline{\W}[\Z]$ can be realized by the semidirect product of the Witt algebra $\overline{\V}[\Z]$
and the $\overline{\V}[\Z]$-module $A_{0,-1}$ of the intermediate series in \cite{KS}, that is,  $\overline{\W}[\Z]\cong\overline{\V}[\Z]\ltimes A_{0,-1}$. What we concern most is the approach of   realizing $\overline{\W}[\Z]$   from    a   truncated  loop-Witt algebra (see, e.g.,
\cite{GLZ,GL}). The detailed description on it will be shown in Section \ref{115eqw}, which is closely associated with the usage  of $\A$-cover theory. The aim of this paper is to present a completely   classification of simple Harish-Chandra  modules over the
high rank  $W$-algebra $W(2,2)$,  which reobtains the
classification result    of classical $W$-algebra $W(2,2)$ (when $k=1$) studied in \cite{LGZ,CLW,GLZ1}.

The paper is organized as follows. In Section $2$, we   introduce some notations and definitions related to the high rank $W$-algebra $W(2,2)$ and Harish-Chandra modules.  We also recall some known  classification theorems  over  several related Lie algebras for later use.
In Section $3$,
  we   give a classification of  simple cuspidal modules  over
the higher rank $W$-algebra $W(2,2)$  in Theorem \ref{the513}.
In Section $4$,  we present  a classification of simple Harish-Chandra modules over  the higher rank $W$-algebra $W(2,2)$ in Theorem \ref{the622}.

Throughout the present article, we denote by $\C$, $\mathbb{R}$, $\Z$, $\N$ and $\Z_+$  the sets of complex numbers, real numbers, integers, nonnegative integers and positive integers, respectively.
 All vector spaces and Lie algebras are over $\C$.  All simple modules are considered to be non-trivial.
For a Lie algebra $\mathfrak{L}$,  we use $U(\mathfrak{L})$ to denote
the universal enveloping algebra.
\section{Preliminaries}
\subsection{The high rank $W$-algebra $W(2,2)$ and its cuspidal module}\label{sec22}
The  {\em high rank $W$-algebra $W(2,2)$} is an infinite dimensional Lie algebra
$$\W[\Z^k]=\bigoplus_{\alpha\in\Z^k}\C L_\alpha \oplus \bigoplus_{\alpha\in\Z^k}\C W_{\alpha} \oplus\C C,$$
which  satisfies  the following  Lie brackets
\begin{equation}\label{def1.1}
\aligned
&[L_\alpha,L_\beta]= (\beta-\alpha)L_{\alpha+\beta}+\delta_{\alpha+\beta,0}\frac{\alpha^{3}-\alpha}{12}C,\\&
[L_\alpha,W_\beta]=(\beta-\alpha)W_{\alpha+\beta}+\delta_{\alpha+\beta,0}\frac{\alpha^{3}-\alpha}{12}C,
\\&
 [W_\alpha,W_\beta]=[\W[\Z^k],C]=0,
\endaligned
\end{equation}
where $\alpha,\beta\in \Z^k,k\in\Z_+$.
Clearly, $\W[\Z^k]$ has an infinite-dimensional Lie subalgebra   $\V[\Z^k]:=\mathrm{span}\{L_\alpha,C\mid \alpha\in\Z^k\}$,  which is called   {\em  high rank  Virasoro algebra}. Note that   $\C C$ is the center of $\W[\Z^k]$.  The quotient algebras $\overline{\W}[\Z^k]=\W[\Z^k]/\C C$ and $\overline{\V}[\Z^k]=\V[\Z^k]/\C C$ are respectively  called {\em   centerless high rank $W$-algebra $W(2,2)$}
and {\em high rank Witt algebra}.
   When $k=1$, we say that
 $\V[\Z]$  and $\W[\Z]$   are respectively    {\em classical Virasoro algebra} and {\em classical  $W$-algebra $W(2,2)$}.
 For any $\alpha\in\Z^k\setminus\{0\}$, we know that $\V[\Z\alpha]$ and $\W[\Z\alpha]$ are respectively  isomorphic
to the classical Virasoro algebra  and classical $W$-algebra $W(2,2)$.

Now we    recall some definitions related to the weight module.   Consider  a non-trivial module $M$ over $\V[\Z^k]$ or $\W[\Z^k]$. We set that  the action  of the central element  $C$ is a scalar $c$.
The module $M$ is said to be  {\em  trivial}  if the action of whole algebra  on $M$  is trivial.
 Denote
$M_\lambda=\{v\in M \mid L_0v=\lambda v\}$, which is called a {\em weight space} of weight $\lambda\in\Z^k$. We call that $M$ is
  a {\em weight module} if $M=\bigoplus_{\lambda\in\Z^k}M_{\lambda}$. Set
$\mathrm{Supp}(M)=\{\lambda\mid  M_{\lambda}\neq0\}$,
which is called the {\em support} (or called the {\em weight set}) of $M$.
The indecomposable weight module $M$ with all
weight spaces one-dimensional is called  the {\em intermediate series module}.

\begin{defi}\rm
Let $M$ be a weight module over $\W[\Z^k]$.
\begin{itemize}
\item[{\rm (1)}]
If
$\mathrm{dim} (M_\lambda)<+\infty$ for all $\lambda\in\mathrm{Supp}(M)$, then  $M$ is  called {\em Harish-Chandra module}.

\item[{\rm (2)}]
If   there exists some  $K\in\Z_+$ such that $\mathrm{dim}(M_\lambda)<K$ for all $\lambda \in \mathrm{Supp}(M)$, then  $M$ is  called {\em cuspidal} (or {\em uniformly bounded}).
\end{itemize}
\end{defi}
We define   a class of  cuspidal $\W[\Z^k]$-modules as follows, which are exactly   intermediate series modules for $\W[\Z^k]$.
\begin{defi}\rm\label{def22}
 For any $g,h \in\C$, the $\W[\Z^k]$-module  $M(g,h;\Z^k)$   has a basis
$\{v_\beta \mid \beta\in\Z^k\}$ and the   $\W[\Z^k]$-action:
\begin{eqnarray*}
&&L_\alpha v_\beta=(g+\beta+h\alpha)v_{\alpha+\beta},
 \ W_\alpha v_\beta=C v_\beta=0.
\end{eqnarray*}
\end{defi}
It is clear that the modules  $M(g,h;\Z^k)$ are isomorphic to the intermediate series modules of $\V[\Z^k]$.
By \cite{SZ},  we see that the modules $M(g,h;\Z^k)$ are reducible if and only if $g\in\Z^k$ and $h\in\{0,1\}$.
 We use
$\overline{M}(g,h;\Z^k)$ to denote the unique non-trivial simple subquotient of $M(g,h;\Z^k)$. Then
$\mathrm{Supp}(\overline{M}(g,h;\Z^k))=g+\Z^k$ or $\mathrm{Supp}(\overline{M}(g,h;\Z^k))=\Z^k\setminus\{0\}$. We also define $\overline{M}(g,h;\Z^k)$ as
intermediate series modules of $\W[\Z^k]$.

\subsection{Generalized highest weight modules}\label{www9998}
In this section,
a general class of Lie algebras are considered. Assume that $\H=\sum_{\alpha\in\Z^k}\H_\alpha$ is a $\Z^k$-graded Lie algebra
such that $\H_0$ is abelian. And the gradation of $\H$ is the root space decomposition
with respect to $\H_0$.

Let   $\g$ be a subgroup of $\Z^k$ such that $\Z^k=\g\oplus\Z\mu$ for
some $\mu\in\Z^k$. We define the   subalgebra  of $\H$ as follows
$$\H_{\g}=\bigoplus_{\alpha\in \g}\H_\alpha,\ \H_{\g}^+=\bigoplus_{\alpha\in \g,m\in\Z_+}\H_{\alpha+m\mu},
\ \H_{\g}^-=\bigoplus_{\alpha\in \g,m\in\Z_+}\H_{\alpha-m\mu}.$$

Let $\K$ be a simple $\H_{\g}$-module. Then  $\K$ can be extended to an  $(\H_{\g}+\H_{\g}^+)$-module
by defining $\H_{\g}^+\K=0$. Now  we can define the {\em generalized Verma module} $V_{\g,\mu,\K}$
for $\H$ as
$$V_{\g,\mu,\K}=\mathrm{Ind}_{\H_{\g}^++\H_{\g}}^{\H}\K
=U(\H)\bigotimes_{U(\H_{\g}+\H_{\g}^+)}\K.$$

It is easy to know that $V_{\g,\mu,\K}$ has a unique simple quotient module for $\H$ and we write it as $P_{\g,\mu,\K}^{\H}$. Then $P_{\g,\mu,\K}^{\H}$ is called a {\em simple
highest weight module}. As far as we know,
the generalized  Verma module  (or generalized highest weight module) was introduced and investigated in some other
references  (see, e.g., \cite{BZ,F,Ma2}).

 Fix a basis of $\Z^k$.
Assume that $M$ is a weight module for $\H$.
Then $M$ is called  {\em dense} if $\mathrm{Supp}(M)=\lambda+\Z^k$ for some $\lambda\in \H_0^*$. On the other hand,
if there exist $\lambda\in\mathrm{Supp}(M), \tau\in \mathbb{R}^k\setminus \{0\}$ and $\beta\in\Z^k$ such that
$$\mathrm{Supp}(M)\subseteq \lambda+\beta+{\Z^k}_{\leq0}^{(\tau)},$$
where ${\Z^k}_{\leq0}^{(\tau)}=\{\alpha\in{\Z^k}\mid(\tau|\alpha)\leq0\}$
and $(\tau|\alpha)$   is the usual inner product in $\mathbb{R}^k$, then  $M$ is called {\em cut}. Obviously, the modules
$P_{\g,\mu,\K}^\H$ defined above are cut modules. If there exist a $\Z$-basis
$\{\epsilon_1,\ldots,\epsilon_k\}$ of ${\Z^k}$ and $K\in \Z_+$ such that $\H_\alpha v=0$ for all
$\alpha=\Sigma_{i=1}^k\alpha_i\epsilon_i$
with $\alpha_i>K, i\in\{1,\ldots,k\}$, then the element $v \in M$ is called a
{\em generalized highest weight vector}.

The following general result of cut
$\H$-modules appeared in Theorem $4.1$ of \cite{MZ}.
\begin{theo}\label{the311}
Let $[\H_\alpha,\H_\beta]=\H_{\alpha+\beta}$ for all $\alpha,\beta\in\Z^k, k\in\Z_+$ with $\alpha \neq\beta$. Assume that
$M$ is a simple weight module over $\H$, which is neither dense nor trivial. If $M$
contains a generalized highest weight vector, then $M$ is a cut module.
\end{theo}

\subsection{The know results}
The classification theorems of simple Harish-Chandra
modules over the classical $W$-algebra $W(2,2)$  and high rank Virasoro algebra will be recalled in
this section.

 The following  result  for   the high rank Virasoro algebra appeared in  \cite{LZ2}.
\begin{theo}\label{411}
Let   $k>1$.  Any non-trivial simple Harish-Chandra module
for $\V[\Z^k]$ is either a module of   intermediate series or isomorphic
to $P_{\g,\mu,\K}^{\V[\Z^k]}$ for some $\mu\in\Z^k\setminus\{0\}$, a subgroup $\g$ of $\Z^k$ with
$\Z^k=\g\oplus\Z \mu$ and a non-trivial simple intermediate series
$\V[\g]$-module $\K$.
\end{theo}

The classification of Harish-Chandra modules over the classical $W$-algebra $W(2,2)$
  was given  in \cite{LGZ}, which was  reobtained   in\cite{GLZ1,CLW} by some new ideas.
  \begin{theo}\label{the44378}
   Any non-trivial simple Harish-Chandra module over
$\W[\Z]$ is either a module  of intermediate series, or a highest/lowest
weight module.
\end{theo}


\section{Cuspidal module}\label{115eqw}
In this section, we   determine the simple cuspidal module for
the higher rank $W$-algebra  $W(2,2)$.

 Let $\Z^k=\bigoplus_{i=1}^{k}\Z \epsilon_i$, where $\epsilon_1,\epsilon_2,\ldots,\epsilon_k$ is a
$\Z$-basis of $\Z^k\subseteq \C$. Given $\alpha\in \Z^k$, we set
$\alpha=\sum_{i=1}^k\alpha_i\epsilon_i$ for $\alpha_i \in \Z$.
For any $\alpha,\beta\in\Z^k$ with $\alpha_i,\beta_i\in\N, i\in\{1,\ldots,k\}$, we denote
$$\alpha^\beta=\alpha_1^{\beta_1}\cdots\alpha_k^{\beta_k}\quad  \mathrm{and}\quad \beta!=\beta_1!\cdots\beta_k!.$$
Conveniently, we denote $\partial:=\frac{d}{dt}$. The high rank  $W$-algebra  $W(2,2)$  can be realized from  truncated high rank loop-Witt algebra $\overline{\V}[\Z^k]\otimes \big(\C[x]/\langle x^2\rangle\big)$ (see, e.g., \cite{LGZ,GLZ1}), namely,
\begin{eqnarray*}
&&L_\alpha=t^{\alpha+1}\partial\otimes1,
\ W_{\alpha}=t^{\alpha+1}\partial\otimes x,
\end{eqnarray*}
where $\alpha\in\Z^k$.
Denote $\A=\mathrm{span}\{t^\alpha\otimes 1 \mid \alpha\in\Z^k\}$, which is a unital associative algebra with multiplication $(t^\alpha\otimes1) (t^\beta\otimes1)=t^{\alpha+\beta}\otimes1$ for $\alpha,\beta\in \Z^k$. For convenience, we write $t^{\alpha+1}\partial=t^{\alpha+1}\partial\otimes1$ and $t^{\alpha}=t^{\alpha}\otimes1$ for $\alpha\in\Z^k$.
\subsection{$\A\overline{\W}[\Z^k]$-module}
We   describe the structure of cuspidal $\overline{\W}[\Z^k]$-modules that admit a compatible action of the commutative  unital algebra $\A$.

\begin{defi}\rm (see \cite{BF1})
A module  $M$ is called an {\em $\A\overline{\mathcal{W}}[\Z^k]$-module} if it is a module for both $\overline{\W}[\Z^k]$ and the commutative unital algebra $\A=\C[t^{\pm1}]\otimes1$
with these two structures being compatible:
\begin{eqnarray}\label{511}
 y(fv)=(yf)v + f(yv)  \quad \mathrm{for}\ f\in \A, y \in \overline{\W}[\Z^k], v\in M.
\end{eqnarray}
\end{defi}

Let $M$  be a weight module over $\A\overline{\W}[\Z^k]$. From \eqref{511}, we  see that the action
of $\A$ is compatible with the weight grading of $M$:
$$\A_\alpha M_\lambda\subset M_{\alpha+\lambda} \quad \mathrm{for}\ \alpha, \lambda\in\Z^k.$$
We suppose that  $\A\overline{\W}[\Z^k]$-module $M$ has a weight space decomposition,  and one of the weight
spaces   is finite-dimensional. According to all non-zero homogeneous elements of $\A$
are invertible,  we know that all weight spaces of $M$ have the same dimension. Then
   $M$ is also a free $\A$-module of a finite rank. It is clear that $\A\overline{\W}[\Z^k]$-module $M$ is
cuspidal (as $\overline{\W}[\Z^k]$-modules).

Assume that $M$ is a cuspidal $\A\overline{\W}[\Z^k]$-module. Let $W=M_g$ for $g\in\Z^k$ and $\mathrm{dim}(W)<\infty$.
From $M$ is a free $\A$-module, we can write
$$M\cong\A\otimes W.$$
\begin{lemm}\label{qas3.3}
Let   $M$ defined as above.
For any $\alpha,n\in\Z^k$, we have $W_\alpha(t^nv)=t^n(W_\alpha v)$ for $v\in M$.
\end{lemm}
\begin{proof}
For any $\beta,n\in\Z^k$, by \eqref{511}, we have
\begin{eqnarray}\label{4rnlk}
W_\beta(t^nv)=(W_\beta t^n)v+t^n(W_\beta  v).
 \end{eqnarray}
Note that  $W_\beta t^{n}=nt^{n+\beta}\otimes x=n(t^{n+\beta}\otimes1)(1\otimes x)$.
 It is clear that $[1\otimes x,\overline{\W}[\Z^k]\oplus \A]=0$.
Then there is a homomorphism of algebras $\chi: 1\otimes x \rightarrow \C$ such
that $1\otimes x $ acts on $M$ as   $\chi(1\otimes x )\in\C$. So    the action of   $W_\beta t^n$ on $M$ can be written as $n\mu t^{n+\beta}$, where $\mu\in\C$.
Now from
\begin{eqnarray*}
&&0
=[W_{\alpha},W_\beta](t^nv)
=n\mu^2(\beta-\alpha)t^{\alpha+\beta+n}v,
\end{eqnarray*}
  we  get $\mu=0$ by taking $n\neq0,\alpha\neq\beta$, namely, $(W_\beta t^n)v=0$.
Putting this into \eqref{4rnlk}, one has $W_\alpha(t^nv)=t^n(W_\alpha v)$ for $v\in M$. The lemma has been proved.
\end{proof}
\begin{remark}\label{rem322}
We note that the $\A\overline{\W}[\Z^k]$-module is a module for the semidirect product Lie
algebras $\overline{\W}[\Z^k]\ltimes\A$ (the action of  $\A$   as a unital commutative associative algebra).  The Lie
brackets between $\overline{\W}[\Z^k]$ and $\A$ are given by $[L_m, t^n]=nt^n, [W_m, t^n]=0$ for $m,n\in\Z^k$.

\end{remark}

For   $m \in\Z^k$,  we consider the following operator
$$\D (m):W \rightarrow W.$$
It   can be defined as the restriction to $W$  of the composition $t^{-m}\circ(t^{m+1}\partial)$ regarded also as an
operator on $M$. Note that $\D(0)=g \mathrm{Id}$.

According to  \eqref{511}, Lemma \ref{qas3.3} and the finite-dimensional operator $\D (m)$, we get the  action on $M$ as follows
\begin{eqnarray}\label{5533}
&&L_m(t^nv)=(t^{m+1}\partial)(t^nv)=nt^{m+n}v+t^{m+n}\D (m)v,
\ W_m(t^nv)=t^n(W_mv),
\end{eqnarray}
where $m,n\in\Z^k,v\in W$.
Based on
   \eqref{def1.1} and  \eqref{5533},
it is easy to derive the Lie
bracket (also see Lemma  $3.2$ in \cite{B1}):
\begin{eqnarray}\label{556677hjj}
[\D (s),\D (m)]=(m-s)\D (s+m)-m\D (m)+s\D (s).
\end{eqnarray}
Next, we   show that $\D (m)$ can be expressed as a polynomial in
$m=(m_1,\ldots,m_k)$.
\begin{theo}\label{533rree}
Assume that  $M$ is a cuspidal $\A\overline{\W}[\Z^k]$-module, $M= \A\otimes W$, where $W=M_g,g\in\Z^k$.
Then the action of $\overline{\mathcal{W}}[\Z^k]$ on $M$ is presented as
\begin{eqnarray*}
&&L_m(t^nv)=nt^{m+n}v+t^{m+n}\D (m)v,
\  W_m(t^nv)=t^n(W_mv),
\end{eqnarray*}
$m,n\in\Z^k,v\in W$,
where the family of operators $\D (m): W\rightarrow W$ can be shown as an $\mathrm{End}(W)$-valued polynomial in $m=(m_1,\ldots,m_k)$ with the constant term $\D (0)=g \mathrm{Id}$, and $\mathrm{Id}$ is the identification endomorphism of $W$.
\end{theo}
\begin{proof} By $m\in\Z^k$, we can write $m=\sum_{{i}=1}^km_{{i}}\epsilon_{{i}}$,  where $m_{{i}}\in\Z$, $\epsilon_1,\epsilon_2,\ldots,\epsilon_k$ is a
$\Z$-basis of $\Z^k$.
According to Theorem $2.2$ in \cite{B1}, we obtain that $\D (m_{{i}}\epsilon_{{i}})$ is a polynomial in $m_i\in \Z$ with
coefficients in $\mathrm{End}(W)$ for all ${{i}}\in \{1,\ldots,k\}$. Now suppose that $\D (\sum_{{{i}}=1}^{{{j}}-1}m_{{i}}\epsilon_{{i}})$
is a polynomial in $\alpha_1,\ldots,\alpha_{{{j}}-1}$ for some $1<{{j}}\leq k$. For $m_{{j}}\in\Z$, it follows from \eqref{556677hjj} that
\begin{eqnarray*}
&&(\sum_{{{i}}=1}^{{{j}}-1}m_{{i}}
\epsilon_{{i}}-m_{{j}}\epsilon_{{j}})\D (\sum_{{{i}}=1}^{{{j}}}m_{{i}}\epsilon_{{i}})
\\&=&[\D (m_{{j}}\epsilon_{{j}}),
\D(\sum_{{{i}}=1}^{{{j}}-1}m_{{i}}\epsilon_{{i}})]
+\sum_{{{i}}=1}^{{{j}}-1}(m_{{i}}\epsilon_{{i}})\D (\sum_{{{i}}=1}^{{{j}}-1}m_{{i}}\epsilon_{{i}})
-(m_{{j}}\epsilon_{{j}})\D (\alpha_{{j}}\epsilon_{{j}}).
\end{eqnarray*}
Consider  $m_{{j}}\neq0$.  Then from the  linearly independence of $\epsilon_1,\ldots,\epsilon_{{j}}$, one has $\sum_{{{i}}=1}^{{{j}}-1}m_{{i}}\epsilon_{{i}}
-m_{{j}}\epsilon_{{j}}\neq0$.  By the induction assumption, we conclude that  $\D (\sum_{{{i}}=1}^{{{j}}}m_{{i}}\epsilon_{{i}})$
is a polynomial in $m_1,\ldots,m_{{{j}}}$, where $1<{{j}}\leq k$.  Choosing  ${{j}}=k$, one can see that
$\D (m)$ is a polynomial in $m_1,\ldots,m_k$. By the definition of operator $\D(m)$, one has $\D (0)=g \mathrm{Id}$ for $g\in\C$. We complete  the proof.
\end{proof}

We can write  $\D (m)$ in  the   form  (also see \cite{B1,BF2,LG})
\begin{eqnarray}\label{dd555}
\sum_{{\widetilde{i}}\in\N^k}\frac{m^{\widetilde{i}}}{{\widetilde{i}}!}D^{({\widetilde{i}})},
\end{eqnarray}
where ${\widetilde{i}}!=\prod_{{\widetilde{j}}=1}^k{\widetilde{i}}_{\widetilde{j}}!$ and  only has a finite number of the  nonzero operators $D^{({\widetilde{i}})}\in \mathrm{End}(W)$. Note that $\D(0)=D^{(\widetilde{0})}$.
For $m,s\in\Z^k$, by \eqref{556677hjj}, we have
\begin{eqnarray}\label{eq55}
&&\nonumber \sum_{{\widetilde{i}},{\widetilde{j}}\in\N^k}\frac{s^{\widetilde{i}}m^{\widetilde{j}}}{{\widetilde{i}}!{\widetilde{j}}!}[D^{({\widetilde{i}})},D^{({\widetilde{j}})}]
\\&=&\nonumber  (\sum_{{\widetilde{i}}\in\N^k}\frac{s^{\widetilde{i}}}{{\widetilde{i}}!}D^{({\widetilde{i}})})(\sum_{{\widetilde{j}}\in\N^k}\frac{m^{\widetilde{j}}}{{\widetilde{j}}!}D^{({\widetilde{j}})})
-(\sum_{{\widetilde{j}}\in\N^k}\frac{m^{\widetilde{j}}}{{\widetilde{j}}!}D^{({\widetilde{j}})})(\sum_{{\widetilde{i}}\in\N^k}\frac{s^{\widetilde{i}}}{{\widetilde{i}}!}D^{({\widetilde{i}})})
\\&=&\nonumber [\D (s),\D (m)]=(s-m)\D (s+m)-s\D (s)+m\D (m)
\\&=& \sum_{l=1}^k\big(\sum_{{\widetilde{i}},{\widetilde{j}}\in\N^k}\frac{s^{\widetilde{i}}m^{\widetilde{j}}}{{\widetilde{i}}!({\widetilde{j}}-\epsilon_l)!}D^{({\widetilde{i}}+{\widetilde{j}}-\epsilon_l)}\big)
-\sum_{l=1}^k\big(\sum_{{\widetilde{i}},{\widetilde{j}}\in\N^k}\frac{s^{\widetilde{i}}m^{\widetilde{j}}}{({\widetilde{i}}-\epsilon_l)!{\widetilde{j}}!}D^{({\widetilde{i}}+{\widetilde{j}}-\epsilon_l)}\big).
\end{eqnarray}
Comparing the   coefficients of $\frac{s^{\widetilde{i}}m^{\widetilde{j}}}{{\widetilde{i}}!{\widetilde{j}}!}$   in \eqref{eq55},  we check that
\begin{eqnarray}\label{22231lk}
[D^{({\widetilde{i}})},D^{({\widetilde{j}})}]=\left\{\begin{array}{llll}
\sum_{l=1}^k({\widetilde{j}}_l-{\widetilde{i}}_l)\epsilon_lD^{({\widetilde{i}}+{\widetilde{j}}-\epsilon_l)} &\mbox{if}\
{\widetilde{i}},{\widetilde{j}}\in\N^k\backslash \{0\},\\[4pt]
0  &\mbox{if}\
{\widetilde{i}}=0\ \mathrm{or}\ {\widetilde{j}}=0.
\end{array}\right.
\end{eqnarray}

From \eqref{22231lk},
the operators $\G=\mathrm{span}\big\{D^{(\widetilde{i})}\mid \widetilde{i} \in \N^k\backslash \{0\}\big\}$ yield a Lie algebra.
For $\widetilde{i}\in\N^k$, let $|\widetilde{i}|=\widetilde{i}_1+\widetilde{i}_2+\cdots+\widetilde{i}_k$.
Denote $$\G_p=\mathrm{span}\{D^{(\widetilde{i})}\mid \widetilde{i}\in\N^k,|\widetilde{i}|-1=p\}\quad \mathrm{for}\ p\in\N.$$
Then  $\G=\bigoplus_{p\in\N}\G_p$  is a $\Z$-graded Lie algebra (also see \cite{LG,BF2,B1}).
From  \eqref{22231lk}, we know that $\G_0=\mathrm{span}\{D^{(\epsilon_{i})}\mid {i}=1,\ldots,k\}$ is a subalgebra of $\G$, whose Lie algebra structure is presented as
$$[D^{(\epsilon_{i})},D^{(\epsilon_{j})}]
=\epsilon_{j}D^{(\epsilon_{i})}-\epsilon_{i}D^{(\epsilon_{j})},$$
where ${i},{j}\in\{1,2,\ldots,k\}$.
It is easy to get  that $[\G_0,\G_0]$ is nilpotent.  Hence,  $\G_0$ is a solvable Lie
algebra. Now  we   define the  following one-dimensional $\G$-module $V(h)=\C v\neq0$ for
any $h\in\C$:
\begin{eqnarray}
D^{(\epsilon_{i})}v=h \epsilon_{i}v \quad \mathrm{for}\ {i}\in\{1,\ldots,k\}.
\end{eqnarray}
The following   lemma  was proved in \cite{LG}.

\begin{lemm}\label{555rree}
Assume that  $T$ and $W$ are    finite-dimensional simple modules over  $\G_0$ and $\G$,  respectively.  Then
\begin{itemize}
\item[\rm(a)] $T\cong V(h)$ for  $h\in\C$.
\item[\rm(b)] $D^{(\widetilde{i})}W=0$ for any $\widetilde{i}\in\N^k$ with $|\widetilde{i}|$ sufficiently large;
\item[\rm(c)] $\G_pW=0$ for all $p\in\Z_+$ and $W\cong V(h)$ as a $\G_0$-module for some $h\in\C$.
\end{itemize}
\end{lemm}

\begin{theo}\label{the566}
Any simple cuspidal $\A\overline{\W}[\Z^k]$-module   is isomorphic to a
module of intermediate series $\overline{M}(g,h;\Z^k)$ for some $g,h\in\C$.
\end{theo}
\begin{proof}
  Let $M$ be a simple cuspidal $\A\overline{\W}[\Z^k]$-module, $M= \A\otimes W$, where $W=M_g,g\in\Z^k$. For $n,\alpha\in\Z^k,v\in W$, based on  Theorem \ref{533rree}, Lemma \ref{555rree} and \eqref{dd555},  we   check that
\begin{eqnarray}\label{qiuy789}
L_\alpha (t^nv)&=&\nonumber nt^{\alpha+n} v+t^{\alpha+n}(\D(\alpha) v)
\\&=&\nonumber t^{\alpha+n}\big((n+\D(0)+\sum_{{\widetilde{i}}\in\N^k\setminus\{0\}}\frac{\alpha^{\widetilde{i}}}
{{\widetilde{i}}!}D^{({\widetilde{i}})})v\big)
\\&=&\nonumber t^{\alpha+n}\big((n+\D(0)+\sum_{i=1}^k\alpha_iD^{({\epsilon_i})})v\big)
\\&=&\nonumber t^{\alpha+n}\big((n+g\mathrm{Id}+\sum_{i=1}^kh\alpha_i\epsilon_i)v\big)
\\&=&(n+g+h\alpha)(t^{\alpha+n}v),
\end{eqnarray}
where $g,h\in\C$.
Using \eqref{qiuy789} and $(\beta-\alpha)W_{\alpha+\beta}(t^nv)=(L_\alpha W_\beta-W_\beta L_\alpha)(t^nv)$, we obtain
\begin{eqnarray}\label{564rf}
(\beta-\alpha)t^n(W_{\alpha+\beta}v)=\beta t^{n+\alpha}(W_\beta v).
\end{eqnarray}
Taking $\beta=0$ in \eqref{564rf}, we immediately get $t^n(W_\alpha v)=0$ for $\alpha\neq0$.
Considering $\alpha=-\beta\neq0$ in \eqref{564rf}  again, one  checks $t^n(W_0v)=0$.
Then we conclude $t^n(W_\alpha v)=0$ for $\alpha,n\in\Z^k$, that is to say, $W_\alpha(t^nv)=0$.
This completes  the proof.
\end{proof}
\subsection{$\A$-cover of a cuspidal $\overline{\W}[\Z^k]$-module}
We first  recall the    definitions of coinduced module and   $\A$-cover (see \cite{BF1}).
\begin{defi}\rm
A module {\em coinduced} from a $\overline{\W}[\Z^k]$-module $M$ is the space $\mathrm{Hom}(\A,M)$ with the actions of $\overline{\W}[\Z^k]$ and $\A$ as follows
\begin{eqnarray*}
&&(\mathfrak{a}\varphi)(f)=\mathfrak{a}(\varphi(f))-\varphi(\mathfrak{a}(f)),
\ (y\varphi)(f)=\varphi(yf),
\end{eqnarray*}
where $\varphi\in\mathrm{Hom}(\A,M), \mathfrak{a}\in \overline{\W}[\Z^k],f,y\in \A$.
\end{defi}
\begin{defi}\rm\label{def5777}
An {\em $\A$-cover} of a cuspidal module $M$ over   $\overline{\W}[\Z^k]$ is an $\A\overline{\W}[\Z^k]$-submodule $$\widehat{M}=\mathrm{span}\{\phi(\mathfrak{a},w)\mid \mathfrak{a}\in\overline{\W}[\Z^k],w\in M\}\subset\mathrm{Hom}(\A,M),$$
where $\phi(\mathfrak{a},w):\A\rightarrow M$ is defined as
$$\phi(\mathfrak{a},w)(f)=(f\mathfrak{a})(w).$$
\end{defi}
The action of $\A\overline{\W}[\Z^k]$ on $\widehat{M}$ is given by
\begin{eqnarray*}
&&\mathfrak{b}\phi(\mathfrak{a},w)=\phi([\mathfrak{b},\mathfrak{a}],w)+\phi(\mathfrak{a},\mathfrak{b}w),
\\&&f\phi(\mathfrak{a},w)=\phi(f\mathfrak{a},w)\quad \mathrm{for}\ \mathfrak{a},\mathfrak{b}\in \overline{\W}[\Z^k],w\in M,f\in \A.
\end{eqnarray*}
Let $$\mathfrak{K}(M)=\left\{\sum_{\alpha\in\Z^k}\mathfrak{a}_\alpha\otimes w_{\alpha}\in\overline{\W}[\Z^k]\otimes M
\ \bigg|\ \sum_{\alpha\in\Z^k}(f\mathfrak{a}_\alpha) w_{\alpha}=0,\quad\forall f\in\A \right\}.$$ Then $\mathfrak{K}(M)$ is an
$\A\overline{\W}[\Z^k]$-submodule of $\overline{\W}[\Z^k]\otimes M$. The $\A$-cover $\widehat{M}$ can also be constructed as a quotient  $\A\overline{\W}[\Z^k]$-module
$$(\overline{\W}[\Z^k]\otimes M)/\mathfrak{K}(M),$$
 where $\overline{\W}[\Z^k]M=M$. Clearly, the following linear map
 \begin{eqnarray*}
\Theta:\quad  & \widehat{M}&\longrightarrow \overline{\W}[\Z^k]M
\\ &\mathfrak{a}\otimes w+\mathfrak{K}(M)&\longmapsto \mathfrak{a}w
\end{eqnarray*} is a $\overline{\W}[\Z^k]$-module epimorphism.

\begin{lemm}{\rm (see \cite{BF1})}\label{lemm6765}
Let $M$ be a  cuspidal module for $\overline{\W}[\Z^k]$. Then there exists $l\in \Z_+$ such that for all
$\alpha,\beta,\gamma\in\Z^k$ the operator $\Omega_{\alpha,\beta}^{(l,\gamma)}=\sum_{{i}=0}^l(-1)^{i}{l\choose {i}}L_{\alpha-{i}\gamma}L_{\beta+{i}\gamma}$ annihilates $M$.
\end{lemm}

\begin{lemm}\label{lemm676225}
Let $M$ be a cuspidal $\overline{\W}[\Z^k]$-module. Then there exists $r\in \Z_+$ such that for all
$\alpha,\beta,\gamma\in\Z^k$ the following two  operators   $$\Omega_{\alpha,\beta}^{(r,\gamma)}=\sum_{{i}=0}^r(-1)^{i}{r\choose {i}}L_{\alpha-{i}\gamma}L_{\beta+{i}\gamma}\quad  \mathrm{and} \quad \widetilde{\Omega}_{\alpha,\beta}^{(r,\gamma)}=\sum_{{i}=0}^r(-1)^{i}{r\choose {i}}W_{\alpha-{i}\gamma}L_{\beta+{i}\gamma}$$  annihilate  $M$.
\end{lemm}
\begin{proof}
Note that    $M$ is  also a    cuspidal module  for $\V[\Z^k]$. According to  Lemma \ref{lemm6765},  there exists
$l\in\Z_+$ such that  $\Omega_{\alpha,\beta}^{(l,\gamma)}M=0$ for all $\alpha,\beta,\gamma\in\Z^k$. It
follows from this  that
\begin{eqnarray}\nonumber
0&=&\Big(\sum_{{i}=0}^l(-1)^{i}{l\choose {i}}\big(L_{\alpha-({i}-1)\gamma}L_{\beta+({i}-1)\gamma}
-2L_{\alpha-{i}\gamma}L_{\beta+{i}\gamma}
+L_{\alpha-({i}+1)\gamma}L_{\beta+({i}+1)\gamma}\big)\Big)M
\\&=&\label{57uyt}\Big(\sum_{{i}=0}^{l+2}(-1)^{i}{l+2\choose {i}}L_{\alpha-({i}-1)\gamma}L_{\beta+({i}-1)\gamma}\Big)M.
\end{eqnarray}
Setting $r=l+2$ in \eqref{57uyt}, one gets  $\Omega_{\alpha,\beta}^{(r,\gamma)}M=0$ for all $\alpha,\beta,\gamma\in\Z^k$.
For any $s\in\Z^k$, from Lemma \ref{lemm6765}, we immediately  get
$[\Omega_{\alpha,\beta}^{(l,\gamma)},W_s]M=0$ for all $\alpha,\beta,\gamma\in\Z^k$.
Now for any $\alpha,\beta,s\in\Z^k$ and  $\gamma\in\Z^k\setminus\{0\}$,  we can compute that
\begin{eqnarray*}
0&=&\Big([\Omega_{\alpha,\beta-\gamma}^{(l,\gamma)},W_{s+\gamma}]
-2[\Omega_{\alpha,\beta}^{(l,\gamma)},W_{s}]
+[\Omega_{\alpha,\beta+\gamma}^{(l,\gamma)},W_{s-\gamma}]
\\&&-[\Omega_{\alpha+\gamma,\beta-\gamma}^{(l,\gamma)},W_{s}]
+2[\Omega_{\alpha+\gamma,\beta}^{(l,\gamma)},W_{s-\gamma}]-[\Omega_{\alpha+\gamma,\beta+\gamma}^{(l,\gamma)},
W_{s-2\gamma}]\Big)M
\\&=&\Big([\sum_{{i}=0}^l(-1)^{i}{l\choose {i}}L_{\alpha-{i}\gamma}L_{\beta-\gamma+{i}\gamma},W_{s+\gamma}]-2[\sum_{{i}=0}^l(-1)^{i}{l\choose {i}}L_{\alpha-{i}\gamma}L_{\beta+{i}\gamma},W_{s}]
\\&&+[\sum_{{i}=0}^l(-1)^{i}{l\choose {i}}L_{\alpha-{i}\gamma}L_{\beta+\gamma+{i}\gamma},W_{s-\gamma}]
-[\sum_{{i}=0}^l(-1)^{i}{l\choose {i}}L_{\alpha+\gamma-{i}\gamma}L_{\beta-\gamma+{i}\gamma},W_{s}]
\\&&+2[\sum_{{i}=0}^l(-1)^{i}{l\choose {i}}L_{\alpha+\gamma-{i}\gamma}L_{\beta+{i}\gamma},W_{s-\gamma}]
-[\sum_{{i}=0}^l(-1)^{i}{l\choose {i}}L_{\alpha+\gamma-{i}\gamma}L_{\beta+\gamma+{i}\gamma},W_{s-2\gamma}]\Big)M
\\&=&\Big(\sum_{{i}=0}^l(-1)^{i}{l\choose {i}}\big(
(s+(2-{i})\gamma-\beta)L_{\alpha-{i}\gamma}W_{\beta+s+{i}\gamma}
\\&&+(s+({i}+1)\gamma-\alpha)W_{\alpha+s+(1-{i})\gamma}L_{\beta+({i}-1)\gamma}
\\&&-2\big((s-\beta-{i}\gamma)L_{\alpha-{i}\gamma}W_{\beta+s+{i}\gamma}+(s-\alpha+{i}\gamma)
W_{\alpha+s-{i}\gamma}L_{\beta+{i}\gamma}\big)
\\&&
+(s-\beta-({i}+2)\gamma)L_{\alpha-{i}\gamma}W_{\beta+s+{i}\gamma}
+(s-\alpha+({i}-1)\gamma)W_{\alpha+s-({i}+1)\gamma}L_{\beta+({i}+1)\gamma}
\\&&
-(s-\beta-({i}-1)\gamma)L_{\alpha+(1-{i})\gamma}W_{\beta+s+({i}-1)\gamma}
-(s-\alpha+({i}-1)\gamma)W_{\alpha+s+(1-{i})\gamma}L_{\beta+({i}-1)\gamma}
\\&&+2\big(
(s-\beta-({i}+1)\gamma)L_{\alpha+(1-{i})\gamma}W_{\beta+s+({i}-1)\gamma}
+(s-\alpha+({i}-2)\gamma)W_{\alpha+s-{i}\gamma}L_{\beta+{i}\gamma}\big)
\\&&
-(s-\beta-({i}+3)\gamma)L_{\alpha+(1-{i})\gamma}W_{\beta+s+({i}-1)\gamma}
\\&&-(s-\alpha+({i}-3)\gamma)W_{\alpha+s-({i}+1)\gamma}L_{\beta+({i}+1)\gamma}\big)\Big)M
\\&=&\Big(2\gamma\sum_{{i}=0}^l(-1)^{i}{l\choose {i}}\big(W_{\alpha+s-({i}-1)\gamma}L_{\beta+({i}-1)\gamma}
-2W_{\alpha+s-{i}\gamma}L_{\beta+{i}\gamma}
+W_{\alpha+s-({i}+1)\gamma}L_{\beta+({i}+1)\gamma}\big)\Big)M
\\&=&\Big(2\gamma\sum_{{i}=0}^{l+2}(-1)^{i}{l+2\choose {i}}W_{\alpha+s-({i}-1)\gamma}L_{\beta+({i}-1)\gamma}\Big)M.
\end{eqnarray*}
Moreover, the module  $M$ can be annihilated by the operator  $\widetilde{\Omega}_{\alpha,\beta}^{(l+2,0)}$.
Then we conclude that $\widetilde{\Omega}_{\alpha,\beta}^{(r,\gamma)}M=\big(\sum_{{i}=0}^r(-1)^{i}{r\choose {i}}W_{\alpha-{i}\gamma}L_{\beta+{i}\gamma}\big)M=0$, where $r=l+2$ and all $\alpha,\beta,\gamma\in\Z^k$. The lemma holds.
\end{proof}

\begin{prop}\label{pro510}
Let $M$ be a cuspidal module over $\overline{\W}[\Z^k]$. Then the $\A$-cover $\widehat{M}$ of $M$ is also  a cuspidal $\A\overline{\W}[\Z^k]$-module.
\end{prop}
\begin{proof}
Let $M_\lambda$ be a weight space with weight $\lambda\in\Z^k$.
For $\alpha\in\Z^k$, we denote
$$\phi(L_\alpha\cup W_\alpha,M_\lambda)=\left\{\phi(L_\alpha,w),\phi(W_\alpha,w)\ \bigg|\  w\in M_\lambda\right\}\subset\widehat{M}.$$
 By considering the weight spaces of $M$,   the space
$\phi(L_\alpha\cup W_\alpha,M_\lambda)$  is finite-dimensional.

Since   $\widehat{M}$ is an $\A$-module,  we see that one of its weight spaces is finite-dimensional.
 For a fixed   weight $\beta\in\Z^k$, we will show that $\widehat{M}_\beta$
is finite-dimensional. Obviously,  the space $\widehat{M}_\beta$
is spanned by the
set $$\Big(\bigcup_{\gamma\in\Z^k}\phi(L_{\beta-\gamma},M_\gamma)\Big)\cup
\Big(\bigcup_{\gamma\in\Z^k}\phi(W_{\beta-\gamma},M_\gamma)\Big).$$
Define a norm on $\Z^k$ as follows
$$\|\alpha\|=\sum_{{i}=1}^k|\alpha_{i}|,$$
where $\alpha=\sum_{{i}=1}^k\alpha_{i}\epsilon_{i}$.
By  Lemma  \ref{lemm676225}, there exists $r\in\N$ such that  for all $\alpha,\beta,\gamma\in\Z^k$  the  following  two operators
$$\Omega_{\alpha,\beta}^{(r,\gamma)}=\sum_{{i}=0}^r(-1)^{i}{r\choose {i}}L_{\alpha-{i}\gamma}L_{\beta+{i}\gamma}\quad \mathrm{and}\quad \widetilde{\Omega}_{\alpha,\beta}^{(r,\gamma)}=\sum_{{i}=0}^r(-1)^{i}{r\choose {i}}W_{\alpha-{i}\gamma}L_{\beta+{i}\gamma}$$
annihilate   $M$, namely, $\Omega_{\alpha,\beta}^{(r,\gamma)}v= \widetilde{\Omega}_{\alpha,\beta}^{(r,\gamma)}v=0$ for all $\alpha,\beta,\gamma\in\Z^k,v\in M$.
Hence, $\Omega_{\alpha,\beta}^{(r,\gamma)}v$ and $\widetilde{\Omega}_{\alpha,\beta}^{(r,\gamma)}v$ are both in $\mathfrak{K}(M)$.
\begin{clai}\label{clai11}
For any $\alpha,\beta\in\Z^k$, $\widehat{M}_{\alpha+\beta}$ is equal to
$$\mathrm{span} \left\{\phi(L_{\alpha+\beta}\cup W_{\alpha+\beta},M_0),
\phi(L_{\alpha-\gamma},M_{\beta+\gamma}),
\phi(W_{\alpha-\gamma},M_{\beta+\gamma})\ \bigg|\  \gamma\neq-\beta,\|\gamma\|\leq\frac{kr}{2}\right\}.
$$
\end{clai}
For all $q\in\Z^k$ and $w\in M_{\beta+q}$, we have $\phi(L_{\alpha-q},w)$ and $\phi(W_{\alpha-q},w)$ in
$\widehat{M}_{\alpha+\beta}$.
Now we  prove this claim by induction on $\|q\|$.
 If $|q_{i}|\leq\frac{r}{2}$ for all ${i}\in\{1,\ldots,k\}$,
the result clears. On the other hand, suppose $|q_{i}|>\frac{r}{2}$  for some ${i}\in\{1,\ldots,k\}$.
 We may assume $q_{i}<-\frac{r}{2}$, and  the other case $q_{i}>-\frac{r}{2}$ can be proved by the similar method.
It is easy to see that  $\|q+{j}\epsilon_{i}\|<\|q\|$ for all ${j}\in\{1,\ldots,r\}$. We only need to give the proof
for $\beta+q\neq0$. From   the action of  $L_0$ on $M_{\beta+q}$ is a nonzero scalar, we can write
 $w=L_0v$, where $v\in M_{\beta+q}$.
We  will verify that
$$\sum_{{j}=0}^r(-1)^{j}{r\choose {j}}\phi(L_{\alpha-q-{j}\epsilon_{i}},L_{{j}\epsilon_{i}}v)= \sum_{{j}=0}^r(-1)^{j}{r\choose {j}}\phi(W_{\alpha-q-{j}\epsilon_{i}},L_{{j}\epsilon_{i}}v)=0$$  in $\widehat{M}$.
 Based on Definition \ref{def5777} and Lemma \ref{lemm676225},
for any $m\in\Z^k$  we deduce
\begin{eqnarray*}
&&\sum_{{j}=0}^r(-1)^{j}{r\choose {j}}\phi(L_{\alpha-q-{j}\epsilon_{i}},L_{{j}\epsilon_{i}}v)(t^m)
\\&=&\sum_{{j}=0}^r(-1)^{j}{r\choose {j}}L_{\alpha+m-q-{j}\epsilon_{i}}L_{{j}\epsilon_{i}}v=\Omega_{\alpha+m-q,0}^{(r,\epsilon_{i})}v=0
\end{eqnarray*}
and
\begin{eqnarray*}
&&\sum_{{j}=0}^r(-1)^{j}{r\choose {j}}\phi(W_{\alpha-q-{j}\epsilon_{i}},L_{{j}\epsilon_{i}}v)(t^m)
\\&=&\sum_{{j}=0}^r(-1)^{j}{r\choose {j}}W_{\alpha+m-q-{j}\epsilon_{i}}L_{{j}\epsilon_{i}}v
=\widetilde{\Omega}_{\alpha+m-q,0}^{(r,\epsilon_{i})}v=0.
\end{eqnarray*}
Thus, one has
\begin{eqnarray}
&&\label{eq588}\phi(L_{\alpha-q},w)=-\sum_{{j}=1}^r(-1)^{j}{r\choose {j}}\phi(L_{\alpha-q-{j}\epsilon_{i}},L_{{j}\epsilon_{i}}v),
\\&&\label{eq599}
\phi(W_{\alpha-q},w)=-\sum_{{j}=1}^r(-1)^{j}{r\choose {j}}\phi(W_{\alpha-q-{j}\epsilon_{i}},L_{{j}\epsilon_{i}}v).
\end{eqnarray}
By induction assumption the right hand sides of \eqref{eq588} and \eqref{eq599}  belong  to $\widehat{M}_{\alpha+\beta}$, and so do
$\phi(L_{\alpha-q},w)$, $\phi(W_{\alpha-q},w)$. This   proves the claim.
Therefore, $\widehat{M}_{\alpha+\beta}$ is finite-dimensional. The proposition follows.
\end{proof}

The Claim \ref{clai11} can also be described as follows.
\begin{remark}
For  $\alpha,\beta,q\in\Z^k$ and $\beta+q\neq0,w\in M_{\beta+q}$, we get
\begin{eqnarray*}
&&\phi(L_{\alpha-q},w)\in \sum_{\|\gamma\|\leq\frac{kr}{2}}\phi(L_{\alpha-\gamma},M_{\beta+\gamma})+\mathfrak{K}(M),
\\&&
\phi(W_{\alpha-q},w)\in\sum_{\|\gamma\|\leq\frac{kr}{2}}\phi(W_{\alpha-\gamma},M_{\beta+\gamma})+\mathfrak{K}(M).
\end{eqnarray*}
\end{remark}
 Now we give a  classification for all simple cuspidal  $\overline{\W}[\Z^k]$-modules.
\begin{theo}\label{the511}
Any simple cuspidal  $\overline{\W}[\Z^k]$-module is isomorphic to a module of intermediate series $\overline{M}(g,h;\Z^k)$ for some $g,h\in\C$.
\end{theo}
\begin{proof}
Assume that  $M$ is a simple cuspidal $\overline{\W}[\Z^k]$-module.   It is clear that $\overline{\W}[\Z^k]M=M$. Then there exist  an $\A$-cover $\widehat{M}$ of $M$ with a surjective homomorphism   $\Theta:\widehat{M}\rightarrow M$.
It follows from Proposition \ref{pro510} that $\widehat{M}$ is a cuspidal $\A\overline{\W}[\Z^k]$-module. Hence, we can  consider the composition series
$$0=\widehat{M}_0\subset\widehat{M}_1\subset\cdots\subset\widehat{M}_{\widehat{c}}=\widehat{M}$$
with the quotients $\widehat{M}_{\widehat{i}+1}/\widehat{M}_{\widehat{i}}$ being simple $\A\overline{\W}[\Z^k]$-modules. Let $d$ be the smallest
integer such that $\Theta(\widehat{M}_d)\neq0$. Then by the  simplicity of  $M$, we obtain
$\Theta(\widehat{M}_d)=M$ and $\Theta(\widehat{M}_{d-1})=0$.
So we get an $\A\overline{\W}[\Z^k]$-epimorphism
$$\overline{\Theta}:\widehat{M}_{d}/\widehat{M}_{d-1}\rightarrow M.$$
Now from   Theorem
\ref{the566}, $\widehat{M}_{d}/\widehat{M}_{d-1}$  is isomorphic to a module of   intermediate series $\overline{M}(g,h;\Z^k)$ for some $g,h\in\C$.
  We complete  the proof.
\end{proof}

Based on the representation
theory of  $\V[\Z^k]$ studied in \cite{S1,S2}, we see that the action of $C$
on any simple cuspidal $\W[\Z^k]$-modules is trivial. Therefore, the category
of simple cuspidal modules over $\W[\Z^k]$  is  equivalent to the category of simple cuspidal modules over $\overline{\W}[\Z^k]$. Then  Theorem \ref{the511} can be described as follows.
\begin{theo}\label{the513}
Let $M$ ba a simple  cuspidal module over $\W[\Z^k]$. Then $M$   is isomorphic to a simple quotient of intermediate
series module $M(g,h;\Z^k)$ for some $g,h\in\C$.
\end{theo}

\section{Non-cuspidal modules}
In this section,  we    determine  the simple weight modules with finite-dimensional weight spaces
which are not cuspidal for the
higher rank $W$-algebras $W(2,2)$. These modules are   called generalized highest weight modules  and defined in Section \ref{www9998}.

The result  of high rank Virasoro algebras of Theorem \ref{411} plays  a key role  in the following proof.
\begin{theo}\label{the611}
Let  $M$ be a simple Harish-Chandra module over $\W[\Z^k]$.     Then $M$   is either
a cuspidal module, or   isomorphic to  some $P_{\g,\mu,\K}^{\W[\Z^k]}$,
where $\mu\in\Z^k\setminus\{0\}$, $\mathfrak{g}$ is a subgroup of $\Z^k$ such that $\Z^k=\mathfrak{g}\oplus\Z\mu$ and $\K$ is a
non-trivial simple cuspidal $\W[\Z^k]_{\mathfrak{g}}$-module.
\end{theo}
\begin{proof}
Suppose that $M$ is not a cuspidal module over  $\W[\Z^k]$. Let us  recall  that $M=\bigoplus_{\alpha\in \Z^k}M_{\alpha}$ where $M_{\alpha}=\{w\in M \mid L_0w=(g+\alpha)w\}$. By Theorem \ref{the44378}, we see that the statement holds for any $\Z^k$ of
  $k=1$.

Now suppose $k\geq2$. View $M$ as a $\V[\Z^k]$-module. Then based on  Theorem \ref{411},
we obtain that  the  action of the central element $C$  on $M$ is  trivial. Hence $M$ can be seen
as a $\overline{\W}[\Z^k]$-module.  We fix a $\Z$-basis of $\Z^k$, which is also suitable for $\Z^k$. For any $\sigma\in \mathbb{R}^k$ and $g \in\Z^k$, we have the inner product $(\sigma|g)$. It follows from $[\overline{\W}[\Z^k]_\alpha, \overline{\W}[\Z^k]_\beta] = \overline{\W}[\Z^k]_{\alpha+\beta}$ that Theorem \ref{the311} can be applied
to $\Z^k$.

Since $M$ is not cuspidal, we can find some rank $k-1$ direct summand
$\widetilde{\Z^k}$ of $\Z^k$ such that $M_{\widetilde{\Z^k}}$ is not cuspidal. Without loss of generality,
we may assume that $\widetilde{\Z^k}$ is spanned by
 $\{\epsilon_1, \epsilon_{2},\ldots,\epsilon_{k}\}\setminus\{\epsilon_{j}\}$, where
 $j\in\{1,2,\ldots,k\}$. Then there exists some
$\widetilde{\alpha}\in \widetilde{\Z^k}$ such that
\begin{eqnarray}\label{ga611}
\mathrm{dim}(M_{-\widetilde{\alpha}})> 2k\big(\mathrm{dim}(M_{\epsilon_{j}})+
\sum_{{i}=1,{i}\neq j}^{k}\mathrm{dim}(M_{\epsilon_{j}+\epsilon_{{i}}})\big).
\end{eqnarray}
For simplicity, we denote $\xi_{j}=\widetilde{\alpha}+\epsilon_{j}$ and $\xi_{{i}}=\widetilde{\alpha}+\epsilon_{j}+\epsilon_{{i}}$ for any ${{i}}\in\{1,2,\ldots,k\}\setminus \{j\}$.
Then it is easy to check that  the linear transformation sending each $\epsilon_{{i}}$ to $\xi_{{i}}$
for any ${{i}}\in\{1,\ldots,k\}$, has determinant $1$ and therefore $\{\xi_{{i}}\mid {i}=1,\ldots,k\}$ is also a $\Z$-basis of $\Z^k$. According to  \eqref{ga611}, there exists some
nonzero element $w\in M_{-\widetilde{\alpha}}$ such that $L_{\xi_{{i}}}w=W_{\xi_{{i}}}w=0$ for all ${{i}}\in\{1,\ldots,k\}$. Thus,
$w$ is a generalized highest weight vector associated with the $\Z$-basis $\{\xi_{{i}} \mid {{i}}=1,\ldots,k\}$.

It is clear that  $M$ is neither dense nor trivial. From Theorem \ref{the311},   there exist
some $\beta\in\Z^k$ and $\tau \in \mathbb{R}^k\setminus \{0\}$ such that $\mathrm{Supp}(M) \subseteq g+\beta+ {\Z^k}^{(\tau)}_{\leq0}$.  Consider $M$ as a $\V[\Z^k]$-module. Then $M$
has a simple non-trivial $\V[\Z^k]$-subquotient, and we denote it by  $\overline{M}^{\V}$, which is not cuspidal. By Theorem \ref{411}, we know that $\overline{M}^{\V}\cong P^{\V[\Z^k]}_{\mathfrak{g},\mu,\mathcal{K}}$ for some
nonzero $\mu\in\Z^k$, subgroup $\mathfrak{g}$ of $\Z^k$ with $\Z^k=\mathfrak{g}\oplus\Z\mu$ and $\mathcal{K}$ being a  simple
 intermediate series module  over $\V[\mathfrak{g}]$. Write
$\Z^k_\tau=\{\alpha \in \Z^k\mid (\tau|\alpha)=0\}$. In particular, we have
$$g -\widetilde{c}\mu+\mathfrak{g} \subseteq \mathrm{Supp}(\overline{M}^{\V}) \subseteq \mathrm{Supp}(M) \subseteq g+\beta+{\Z^k}^{(\tau)}_{\leq0}$$
for sufficiently large $\widetilde{c}\in \Z_+$.  This gives $\mathfrak{g}=\Z^k_\tau$ and $(\tau|\mu)>0$.

We set that $\widetilde{c}_0\in\Z$ is the maximal number such that $\mathcal{K}=M_{g+\widetilde{c}_0\mu+\mathfrak{g}}\neq0$. Hence,
$\W[\Z^k]^+\mathcal{K}=0$. Then it follows from the simplicity  of $\W[\Z^k]$-module $M$   that the simple $\W[\g]$-module  $\mathcal{K}$   and  $M=P^{\W[\Z^k]}_{\mathfrak{g},\mu,\mathcal{K}}$. At last,  note that
$\mathcal{K}$ is non-trivial and cuspidal. This proves the theorem.
\end{proof}
Based on  Theorems \ref{the513} and  \ref{the611}, we give a classification of simple Harish-Chandra modules over  the higher rank $W$-algebra $W(2,2)$.
\begin{theo}\label{the622}
Assume that $M$ is a non-trivial simple Harish-Chandra module over $\W[\Z^k]$ for some
$k\in\Z_+$.
\begin{itemize}
\item[\rm(1)] If $M$ is  cuspidal, then $M$ is isomorphic to some $\overline{M}(g,h;\Z^k)$
for some $g,h\in \C$;
\item[\rm(2)] If $M$ is not cuspidal, then $M$ is isomorphic to
$P^{\W[\Z^k]}_{\mathfrak{g},\mu,\mathcal{K}}$ for some $\mu\in\Z^k\setminus\{0\}$, a subgroup  $\mathfrak{g}$ of $\Z^k$ with $\Z^k=\mathfrak{g}\oplus\Z{\mu}$ and  a non-trivial simple intermediate series   $\W[\mathfrak{g}]$-module  $\mathcal{K}$.
\end{itemize}
\end{theo}
Note that  Theorem \ref{the44378} is a special case of Theorem \ref{the622} for  $k=1$.
\section*{Acknowledgements}
This work was supported by the National Natural Science Foundation of China
(No. 12171129).

\bigskip

Haibo Chen
\vspace{2pt}

  1. School of  Statistics and Mathematics, Shanghai Lixin University of  Accounting and Finance,   Shanghai
201209, China

\vspace{2pt}
2. Department of Mathematics, Jimei University, Xiamen, Fujian 361021, China

\vspace{2pt}
Hypo1025@163.com

\end{document}